\documentclass[draft,reqno]{amsart}

\usepackage{amssymb}
\usepackage{epic}
\usepackage{mathrsfs}
\usepackage{hhline}

\usepackage{pzccal}

\newcommand{\bbold}{\mathbb}
\newcommand{\cal}{\mathcal}

\def\R { {\bbold R} }

\def\Z { {\bbold Z} }

\def\N { {\bbold N} }

\renewcommand\epsilon{\varepsilon}
\renewcommand\rho{\varrho}

\def \<{\langle}
\def \>{\rangle}

\def \hat {\widehat}

\def \supp {\operatorname{supp}}

\def \RxLE {\R [[  x^{\R}  ]] ^{\operatorname {LE}}  }

\def \((  {(\!(}
\def \)) {)\!)}

\DeclareMathSymbol{\precequ}{\mathrel}{symbols}{"16}
\DeclareMathSymbol{\succequ}{\mathrel}{symbols}{"17}

\newtheorem{theorem}{Theorem}[section]
\newtheorem{lemma}[theorem]{Lemma}
\newtheorem{prop}[theorem]{Proposition}
\newtheorem{cor}[theorem]{Corollary}

\theoremstyle{definition}

\theoremstyle{remark}
\newtheorem*{example}{Example}
\newtheorem*{examples}{Examples}

\newtheorem*{notation}{Notation}

\newtheorem*{question}{Question}

\newcommand{\abs}[1]{\lvert#1\rvert}

\def \fM {{\mathfrak M}}
\def \fd {{\mathfrak d}}
\def \fm {{\mathfrak m}}
\def \fn {{\mathfrak n}}

\def \CM {{{C}}[[ \fM ]] }

\def \RM {{{\R}}[[ \fM ]] }

\def \id{\operatorname{id}}

\let\oldi\i
\let\oldj\j
\renewcommand\i{\relax\ifmmode{\boldsymbol{i}}\else\oldi\fi}
\renewcommand\j{\relax\ifmmode{\boldsymbol{j}}\else\oldj\fi}

\renewcommand\leq{\leqslant}
\renewcommand\geq{\geqslant}
\renewcommand\preceq{\preccurlyeq}

\renewcommand\frak{\mathfrak}

\DeclareMathAlphabet{\mathbf}{OML}{cmm}{b}{it}

\DeclareFontFamily{U}{fsy}{}
\DeclareFontShape{U}{fsy}{m}{n}{<->s*[.9]psyr}{}
\DeclareSymbolFont{der@m}{U}{fsy}{m}{n}
\DeclareMathSymbol{\der}{\mathord}{der@m}{182}

\DeclareSymbolFont{der@m}{U}{fsy}{m}{n}
\DeclareMathSymbol{\derdelta}{\mathord}{der@m}{100}

\DeclareSymbolFont{imag@m}{OT1}{cmr}{m}{ui}
\DeclareMathSymbol{\imag}{\mathord}{imag@m}{105}

\DeclareFontFamily{OMS}{smallo}{}
\DeclareFontShape{OMS}{smallo}{m}{n}{<->s*[.7]cmsy10}{}
\DeclareSymbolFont{smallo@m}{OMS}{smallo}{m}{n}
\DeclareMathSymbol{\smallo}{\mathord}{smallo@m}{79}

\DeclareFontFamily{OMS}{largerdot}{}
\DeclareFontShape{OMS}{largerdot}{m}{n}{<->s*[.8]cmsy10}{}
\DeclareSymbolFont{largerdot@m}{OMS}{largerdot}{m}{n}
\DeclareMathSymbol{\largerdot}{\mathord}{largerdot@m}{15}

\DeclareMathSymbol{\llambda}{\mathord}{der@m}{108}
\DeclareMathSymbol{\rrho}{\mathord}{der@m}{114}

\newcommand\st{\operatorname{st}}
\newcommand\ch{\operatorname{ch}}

\DeclareMathAlphabet{\mathpzc}{OT1}{pzcm}{mb}{it}

\begin{document}

\footnotetext{May 2012}

\title[Transseries and asymptotic fields]{Transseries and Todorov-Vernaeve's asymptotic fields}

\author{Matthias Aschenbrenner}
\email{matthias@math.ucla.edu}
\address{Department of Mathematics \\
University of California, Los Angeles \\ 
Box 951555 \\
Los Angeles, CA 90095-1555, U.S.A.}

\author{Isaac Goldbring}
\email{isaac@math.ucla.edu}
\address{Department of Mathematics \\
University of California, Los Angeles \\ 
Box 951555 \\
Los Angeles, CA 90095-1555, U.S.A.}

\begin{abstract}
We study the relationship between fields of transseries 
and residue fields of convex subrings of non-standard extensions of the real numbers. This was motivated by  a question of Todorov and Vernaeve, answered in this paper.
\end{abstract}

\maketitle

\noindent
In this note we answer a question by Todorov and Vernaeve (see, e.g., \cite{TV09}) concerning the relationship between the field of logarithmic-exponential series from~\cite{DMM01} and the residue field of a certain convex subring of a non-standard extension of the real numbers, introduced in \cite{TV12} in connection with a non-standard approach to Colombeau's theory of generalized functions. 
The answer to this question can almost immediately be deduced from well-known (but non-trivial) results about o-minimal structures. It should therefore certainly be familiar to logicians working in this area, but perhaps less so to those in non-standard analysis, and hence may be worth recording.

\medskip
\noindent
We begin by explaining the question of Todorov-Vernaeve.
Let ${}^*\R$ be a non-standard extension of $\mathbb R$.
Given $X\subseteq\R^m$ we denote the non-standard extension of $X$ by ${}^*X$, and given also a map $f\colon X\to\R^n$, by abuse of notation we  denote the non-standard extension of $f$ to a map ${}^* X\to {}^*\R^n$ by the same symbol~$f$.

\medskip
\noindent
Let $\mathcal O$ be a convex subring of ${}^*\R$. Then $\mathcal O$ is a valuation ring of ${}^*\R$, with maximal ideal
$$\mathcal o := \{ x\in {}^*\R: \text{$x=0$, or $x\neq 0$ and $x^{-1}\notin\mathcal O$} \}.$$
We denote the residue field $\mathcal O/\mathcal o$ of $\mathcal O$ by $\widehat{\mathcal O}$, with natural surjective morphism 
$$x\mapsto\widehat{x}:=x+\mathcal o\colon\mathcal O\to\widehat{\mathcal O}.$$ 
The ordering of ${}^*\R$ induces an ordering of $\widehat{\mathcal O}$ making $\widehat{\mathcal O}$ an ordered field and $x\mapsto\widehat{x}$ order-preserving. By standard facts from real algebra \cite{KW}, $\widehat{\mathcal O}$ is real closed.
Residue fields of convex subrings of ${}^*\R$ are called ``asymptotic fields'' in \cite{TV12} (although this terminology is already used with a different meaning elsewhere~\cite{ADH11}).
The collection of convex subrings of ${}^*\R$ is linearly ordered by inclusion, and the
smallest convex subring of ${}^*\R$ is  
$${}^*\R_{\operatorname{fin}} = \{ x\in {}^*\R: \text{$\abs{x}\leq n$ for some $n$}\},$$
with maximal ideal
$${}^*\R_{\operatorname{inf}} = \left\{ x\in {}^*\R: \text{$\abs{x}\leq \textstyle\frac{1}{n}$ for all $n>0$}\right\}.$$
The inclusions $\R\to{}^*\R_{\operatorname{fin}}\to\mathcal O$ give rise to a field embedding $\R\to\widehat{\mathcal O}$, by which we identify~$\R$ with a subfield of~$\widehat{\mathcal O}$. 
In the case $\mathcal O={}^*\R_{\operatorname{fin}}$ we have $\widehat{\mathcal O}=\R$, and $\widehat{x}$ is the standard part of $x\in{}^*\R_{\operatorname{fin}}$, also denoted in the following by $\st(x)$.

\medskip
\noindent
Let now  $\xi\in{}^*\R$ with $\xi>\R$ and let $\mathcal E$ be the smallest convex subring of ${}^*\R$ containing all iterated exponentials
$\xi, \exp \xi, \exp\exp\xi,\dots$ of $\xi$,  that is,
$$\mathcal E = \big\{x\in {}^*\R: \text{$\abs{x} \leq \exp_n(\xi)$ for some $n$} \big\},$$
where $\exp_0(\xi)=\xi$ and $\exp_n(\xi)=\exp(\exp_{n-1}(\xi))$ for $n>0$.
The  maximal ideal of~$\mathcal E$ is
$$\mathcal e = \left\{x\in {}^*\R: \text{$\abs{x} \leq \textstyle\frac{1}{\exp_n(\xi)}$ for all $n$} \right\},$$
with residue field $\widehat{\mathcal E}=\mathcal E/\mathcal e$.
Note that the definition of $\widehat{\mathcal E}$ depends on the choice of ${}^*\R$ and~$\xi$, which is suppressed in our notation; in \cite{TV12,TV09}, $\widehat{\mathcal E}$ is denoted by $\widehat{\mathcal E_\rho}$ where $\rho=1/\xi$.

\medskip
\noindent
By an {\bf exponential field} we mean an ordered field $K$ equipped with an exponential function on $K$, i.e., an isomorphism $f\mapsto \exp(f)$ between the ordered additive group of $K$ and the ordered multiplicative group $K^{>0}$ of positive elements of $K$. We often write $e^f$ instead of $\exp(f)$, and the inverse of $\exp$ is denoted by $\log\colon K^{>0}\to K$.
It is shown in \cite{TV09} (see also Section~\ref{sec:differentiability in residue fields} below) that there are mutually inverse group morphisms
$\exp\colon\widehat{\mathcal E}\to{\widehat{\mathcal E}}^{>0}$ and $\log\colon{\widehat{\mathcal E}}^{>0}\to\widehat{\mathcal E}$ 
with
$$\exp(\widehat{f}) = \widehat{\exp(f)},\quad \log(\widehat{g}) = \widehat{\log(g)}\qquad\text{for all $f,g\in\mathcal E$, $\widehat{g}>0$.}$$
Thus $(\widehat{\mathcal E},\exp)$ is an exponential field containing the real exponential field~$(\R,\exp)$ as an exponential subfield. 

\medskip
\noindent
Let now $\R[[x^\Z]]$ be the field of Laurent series in descending powers of $x$ with coefficients in~$\R$, made into an ordered field such that 
$x>\R$.
Then $\R[[x^\Z]]$ does not admit an exponential function, but one can embed $\R[[x^\Z]]$ into a large real closed ordered
field, whose elements are formal series (with monomials coming from some ordered abelian group extending $x^\Z$), which does carry an exponential function extending the  exponential function on~$\R$.
There are several ways of performing such a construction, leading
to different (non-isomorphic) exponential fields  of {\bf transseries.} One such construction leads to the exponential field $\RxLE$
of logarithmic-exponential series (or LE-series), introduced in \cite{DG,Ecalle} and further studied in \cite{DMM01,vdH97,vdH}. Another construction results in a  properly larger exponential field, first defined in \cite{vdH97} (and also employed in \cite{ADH05}), which is sometimes called the field of exponential-logarithmic series (or EL-series) and denoted here by~$\mathbb T$.

\medskip
\noindent
The following was asked in \cite{TV09}:

\begin{question}
Is there an embedding $\RxLE\to\widehat{\mathcal E}$ (of fields) which is the identity on~$\R$?
\end{question}

\noindent
Our aim in this note is to give a positive answer to this question,  under a sensible extra hypothesis on ${}^*\R$; in fact, we show that then the structure on $\widehat{\mathcal E}$ is even richer than suggested by the question above:

\begin{theorem}\label{thm:main theorem}
Suppose ${}^*\R$ is $\aleph_1$-saturated. Then there exists an elementary embedding $\mathbb T\to\widehat{\mathcal E}$ of exponential fields which is the identity on $\R$ and sends~$x$ to~$\xi$.
\end{theorem}

\noindent
Note that the hypothesis on ${}^*\R$ is satisfied automatically if ${}^*\R$ is obtained (as usual) as an ultrapower of $\R$. Theorem~\ref{thm:main theorem} is a special case of a general embedding result stated and proved in Section~\ref{sec:T-Convexity} below. As a consequence of this result, the embedding $\mathbb T\to\widehat{\mathcal E}$ in Theorem~\ref{thm:main theorem} can be  chosen to additionally respect the natural structure on $\mathbb T$, respectively $\widehat{\mathcal E}$, coming from the restricted real analytic functions; see Theorem~\ref{thm:main theorem, 2}. This embedding  can further be extended to an embedding, also respecting the restricted analytic structure, of the maximal immediate extension of $\mathbb T$ into~$\widehat{\mathcal E}$ (where $\mathbb T$ is equipped with the valuation whose valuation ring is the convex hull of $\R$ in $\mathbb T$); see Corollary~\ref{cor:main theorem}. 
We should mention that fields of transseries, even larger than~$\mathbb T$, are constructed in \cite{vdH97, Schmeling}, which also carry exponential functions which (presumably) make them elementary extensions of $(\R,\exp)$ of countable character (see Section~\ref{sec:ctble char}), so the embedding theorem in Section~\ref{sec:T-Convexity} could be applied to produce embeddings of such fields of transseries into $\widehat{\mathcal E}$.
(For another application of our embedding result, to surreal numbers, see Corollary~\ref{cor:surreal} below.)

\medskip
\noindent
The ordered field $\mathbb T$ has more structure than that dealt with in Theorem~\ref{thm:main theorem, 2}.  
For example, it
comes equipped with a natural derivation $\frac{d}{dx}$ making it an $H$-field in the sense of \cite{AvdD} (see, e.g., \cite{ADH05}).  By Lemma~\ref{lem:T models Tanexp} and the remarks following Lemma~\ref{lem:induced} below, $(\R,\exp)$ is an elementary substructure of both  $(\widehat{\mathcal E},\exp)$ and $(\mathbb T,\exp)$, so if $\widehat{\mathcal E}\neq\R$ then
also $(\widehat{\mathcal E},\exp,\R) \equiv (\mathbb T,\exp,\R)$ by \cite{vdDL}.
Thus, if the exponential field $(\widehat{\mathcal E},\exp)$ equipped with a predicate for its 
subfield $\R$ happens to be splendid in the sense of \cite[Section~10.1]{Hodges} (e.g., if it is saturated), then there also exists a derivation $\partial$ on $\widehat{\mathcal E}$ making this ordered field an $H$-field with constant field $\R$, such that $(\widehat{\mathcal E},\exp,\R,\partial)$ is elementarily equivalent to $(\mathbb T,\exp,\R,\frac{d}{dx})$.  
This raises the following question, not addressed in the present paper:

\begin{question}
Suppose  ${}^*\R$ is $\aleph_1$-saturated. 
Is there a derivation on $\widehat{\mathcal E}$ making this ordered field an $H$-field with constant field $\R$, and an elementary 
embedding $\mathbb T\to\widehat{\mathcal E}$ of exponential differential fields which is the identity on $\R$ and sends~$x$ to~$\xi$?
\end{question}

\subsection*{Organization of the paper}
We begin by recalling the construction of the exponential field $\mathbb T$ in Section~\ref{sec:T}. 
In Section~\ref{sec:character etc} we then show an embedding statement (Lemma~\ref{lem:embedd}) which is used
in the following Section~\ref{sec:T-Convexity} to prove the main theorem stated above. Section~\ref{sec:differentiability in residue fields} finally contains some remarks on the functions on residue fields of convex subrings of ${}^*\R$ induced by real $C^\infty$-functions.

\subsection*{Conventions and notations}
We let $m$ and $n$ range over the set $\N:=\{0,1,2,\dots\}$ of natural numbers.
We use ``ordered set'' synonymously with ``linearly ordered set.''
Let $S$ be an ordered set.  We let $S_{\pm\infty}$ be the set 
$S\cup\{+\infty,-\infty\}$, where $+\infty$ and $-\infty$ are distinct and not in~$S$, equipped with the extension of the ordering on $S$ to $S_{\pm\infty}$ satisfying $-\infty<s<+\infty$ for all $s\in S$, and we also let $S_\infty$ be the ordered subset $S\cup\{+\infty\}$ of $S_{\pm\infty}$.
An interval in $S$ is a subset of $S$ of the form
$$(a,b) = \{ x\in S : a<x<b \} \qquad (a,b\in S_{\pm\infty},\ a<b).$$
If $S'$ is an ordered set extending $S$, then for  $a,b\in S_{\pm\infty}$ with $a<b$ we also let
$$(a,b)_{S'} := \{ x\in S' : a<x<b \}.$$
We equip $S$ with the order topology, with a subbasis of open sets given by the  intervals, and we equip each cartesian power $S^n$ of $S$ with the corresponding product topology.

\subsection*{Acknowledgements} The first-named author was partially supported by NSF grant DMS-0969642
and the second-named author was partially supported by NSF grant DMS-1007144.
Both authors would like to thank Dave Marker for the argument in the example in Section~\ref{sec:differentiability in residue fields}.

\section{Constructing $\mathbb T$}\label{sec:T}

\noindent
In this section we explain the construction of the transseries field $\mathbb T$.

\subsection{Well-based series}
Let $\fM$ be an ordered abelian group, written {\em multiplicatively,}
with identity $1$. We refer to the elements of $\fM$ as {\bf monomials},
write the ordering on $\fM$ as $\preccurlyeq$, and put
$\fm\prec\fn$ if $\fm\preccurlyeq\fn$ and $\fm\neq\fn$, for $\fm,\fn\in\fM$.
Let $C$ be a field.  
A {\bf well-based series} with coefficients in $C$ and
monomials from $\fM$ is a mapping
$f\colon\fM\to C$ whose {\bf support} 
$$\supp f := \bigl\{\fm\in\fM: f(\fm)\neq 0\bigr\}$$
is {\bf well-based}, that is, 
there exists no infinite
sequence $\fm_1,\fm_2,\dots$ of monomials
in $\supp f$ with $\fm_1\prec \fm_2\prec\cdots$. 
We put $f_\fm=f(\fm)$, and we usually write $f$ as a formal sum
$$f=\sum_{\fm\in\fM} f_\fm \fm.$$
We denote the set of well-based series with coefficients in $C$ and monomials from
$\fM$ by $\CM$. It was first noted by Hahn (1907) that
$\CM$ is a field with respect to the natural addition and multiplication of
well-based series: 
$$f+g = \sum_{\fm\in\fM} (f_\fm +g_\fm)\, \fm, \qquad
f\cdot g = \sum_{\fm\in\fM} \left(\sum_{{\mathfrak u}\cdot 
{\mathfrak v}=\fm} f_{\mathfrak u}\cdot g_{\mathfrak v}\right)\fm.$$
We call $\CM$ the {\bf field of well-based series} (or the {\bf Hahn field}) with coefficients in $C$ and
monomials from $\fM$.
It contains $C$ as a subfield,
identifying $c\in C$ with the series $f\in \CM$ such that $f_1=c$ and
$f_\fm=0$ for $\fm \ne 1$.
Given $f\in\CM$ we call $f_1\in C$ the {\bf constant term} of $f$.

\begin{example}
Let $R$ be an ordered subgroup of the ordered additive group
$\R$, and let $\fM=x^R$ be a multiplicative copy of $R$, with 
order-preserving isomorphism $r\mapsto x^r\colon R\to x^R$. Then 
$\CM=C[[x^R]]$ is a Hahn field with coefficients in~$C$ and
monomials of the form $x^r$, $r\in R$. Taking $R=\Z$ we obtain 
the field of formal Laurent series in descending powers of $x$ 
with coefficients in~$C$.  
\end{example}

\noindent
The support of any non-zero $f\in\CM$, being well-based, 
has a maximal element
(with respect to $\preccurlyeq$), called the {\bf dominant monomial} 
$$\fd(f)=\max\supp f$$ of $f$. We also set $\fd(0):=0$, and 
extend $\preccurlyeq$ to a linear ordering on $\{0\}\cup \fM$
by declaring $0\preccurlyeq \fm$ for $\fm\in \fM$. This linear
ordering is extended to a binary relation on $\CM$ by
$$f\preccurlyeq g \qquad:\Longleftrightarrow\qquad \fd(f) 
\preccurlyeq \fd(g).$$ 
Every $f\in\CM$
can be decomposed as $$f=f^{\succ}+f_1+f^{\prec}$$ 
where $f_1\in C$ is the coefficient of $1\in\mathfrak M$ in $f$, and
\begin{align*}
f^{\succ} &= \sum_{\fm\succ 1} f_\fm \fm \qquad\text{ ({\bf infinite part} 
of $f$)},\\
f^{\prec} &= \sum_{\fm\prec 1} f_\fm \fm 
\qquad\text{ ({\bf infinitesimal part} of $f$).}
\end{align*}
This gives rise to a decomposition of $\CM$ into a direct sum of $C$-vector
spaces:
$$\CM = \CM^{\succ} \oplus C\oplus \CM^{\prec}$$
where 
\begin{align*}
\CM^{\succ} &:= \bigl\{f\in\CM: \text{$\fm\succ 1$ for all 
$\fm\in\supp f$}\bigr\}, \\
\CM^{\prec} &:= \bigl\{f\in\CM : \text{$\fm\prec 1$ for all 
$\fm\in\supp f$}\bigr\}.
\end{align*}
Let $\Gamma$ be an additively written copy of $\mathfrak M$, with isomorphism $\mathfrak m\mapsto v\mathfrak m\colon\mathfrak M\to\Gamma$, equipped with the ordering $\leq$ making this isomorphism decreasing: $\mathfrak m\preceq \mathfrak n \Longleftrightarrow v\mathfrak m\geq v\mathfrak n$, for all $\mathfrak m,\mathfrak n\in\mathfrak M$.
Then the map
$$v\colon \CM\to\Gamma_\infty, \qquad  vf=\begin{cases} v(\mathfrak d(f)) & \text{if $f\neq 0$} \\
\infty & \text{if $f=0$}\end{cases}$$
is a valuation on the field $\CM$, with valuation ring $\mathcal O:=\CM^{\preceq}=C\oplus \CM^{\prec}$, whose maximal ideal is $\mathcal o:=\CM^{\prec}$. The map sending $f\in\mathcal O$ to its constant term $f_1$ is a surjective ring morphism $\mathcal O\to C$, with kernel $\mathcal o$, and hence induces an isomorphism $\mathcal O/\mathcal o\to C$ between the residue field of $\mathcal O$ and the coefficient field $C$.

\subsection{Ordering of well-based series}
In the following we are mainly interested in the case where $C$ is equipped with an ordering making $C$ an ordered field.
Then we make $\CM$ 
into an ordered field extension of $C$ as follows:
for $0\neq f\in\CM$ define
$$f>0 \quad :\Longleftrightarrow\quad f_{\fd(f)}>0.$$
Note that for $f,g\in\CM$,
$$
f\preccurlyeq g \qquad\Longleftrightarrow\qquad \abs{f}\leq c\abs{g}
\text{ for some $c\in C^{>0}$.}
$$
The valuation ring $\mathcal O$ of $\CM$ is a convex subring of $\CM$.

\medskip
\noindent
Suppose $\fM_1$ and $\fM_2$ are subgroups of $\fM$ with $\fM_1$ 
convex in $\fM$,
$\fM=\fM_1\cdot\fM_2$, and $\fM_1\cap\fM_2=\{1\}$. Then we have an 
ordered field isomorphism
$$f =\sum_{\fm\in\fM} f_\fm\fm \mapsto \sum_{\fm_2\in\fM_2} 
\left(\sum_{\fm_1\in\fM_1}f_{\fm_1\fm_2}\fm_1\right)\fm_2 
 \colon  \quad C[[\fM]] \to (C[[\fM_1]])[[\fM_2]] $$
which is the identity on $C$. 
We identify $\CM$ and
$(C[[\fM_1]])[[\fM_2]]$ via this isomorphism whenever convenient.

\subsection{Analytic structure on Hahn fields}\label{sec:analytic on Hahn fields}
Let $\mathfrak M$ be a multiplicatively written ordered abelian group and $K=\R[[\mathfrak M]]$. Every analytic function
$f\colon (a,b)\to\R$ on an interval, where  $a,b \in\R_{\pm\infty}$, $a<b$, extends naturally to
$\widehat{f}\colon (a,b)_K\to K$, where
$$(a,b)_K=\bigl\{g\in K:a<g<b\bigr\}=\bigl\{c+\varepsilon:
c\in (a,b),\ \varepsilon\in K^{\prec 1}\bigr\},$$
given by
$$\hat{f}(c+\varepsilon) := 
\sum_{n=0}^\infty \frac{f^{(n)}(c)}{n!}\varepsilon^n
\qquad\text{for $c\in (a,b)$ and
$\varepsilon\in K^{\prec 1}$.}$$ 
For this, one needs to show that the infinite sum on the
right-hand side makes
sense in $K=\RM$; see, e.g., \cite[p.~64]{DMM01}. For example, the real
exponential function $c\mapsto e^c\colon\R\to\R^{>0}$ extends to the function
\begin{equation}\label{eq:extension exponential}
c+\varepsilon \mapsto \exp(c+\varepsilon)=e^{c+\varepsilon}:=
e^c\cdot\sum_{n=0}^\infty \frac{\varepsilon^n}{n!}
\qquad\text{($c\in\R$ and
$\varepsilon\in K^{\prec 1}$)}
\end{equation}
from $K^{\preceq 1}=\R\oplus K^{\prec 1}$ to $K^{>0}$.
For 
$f,g\in K^{\preccurlyeq 1}$, we have
$$\exp ( f )
\geqslant 1 \Leftrightarrow f \geqslant 0,\quad\exp ( f ) \geqslant f + 1,
\quad\text{and}\quad
\exp ( f + g ) = \exp ( f ) \exp ( g ).$$ 
%(If $\mathfrak{M} \ne \{1 \}$, then there is no total exponential 
%function on $\R [[ \mathfrak{M} ]]$.)
Thus $\exp$ is injective with image 
$$K^{>0,\asymp 1}=\bigl\{g\in K: g > 0,\ \mathfrak{d}(g)=1\bigr\}$$ and inverse
$$\log \colon K^{>0,\asymp 1} \to K^{\preccurlyeq 1}$$
given by $$\log g := \log c + \log (1+\varepsilon)\qquad\text{ for $g=c(1+\varepsilon)$,
$c \in \R^{> 0}$, $\varepsilon \prec 1$,}$$ 
where $\log c$ is the usual natural 
logarithm of the positive real number $c$ and 
$$\log ( 1 + \varepsilon ):=\sum_{n = 1}^{\infty} \frac{( - 1 )^{n+1}}{n}
\varepsilon^n.$$
In a similar way, every real-valued function $f\colon I^n\to\R^n$, analytic on an open neighborhood of the cube $I^n=[-1,1]^n$ in $\R^n$, extends naturally to a function $\widehat{f}$ on the corresponding cube $I(K)^n$ in $K^n$ (taking values in~$K$).
%Moreover, the ordered field $\R[[\mathfrak M]]$ equipped with the extensions $\widehat{f}$ of all such functions $f$ (for varying $n$), is an elementary extension of the ordered field $\R$ equipped with the functions $f\colon I_n\to\R^n$; see Section~\ref{} below. \marginpar{Needs more details.}

\medskip
\noindent
One drawback of $K=\RM$ is that this ordered field does not support a (total) exponential
function if $\fM\ne \{1\}$, as shown in \cite{KKS}. 
Nonetheless, one can extend  $x^\Z$ to a
large ordered multiplicative group $\mathfrak T$ and $\R[[x^{\Z}]]$ to a real closed {\it subfield}\/ $\mathbb T$ of the well-based series field $\R[[\mathfrak T]]$, such that
$\mathfrak T \subseteq \mathbb T^{>0}$ 
inside $\R[[\mathfrak T]]$, 
and such that the usual
exponential function on $\R$ extends to
an exponential function on $\mathbb T$. 
In the following we outline the construction of such an exponential field~$\mathbb T$.

\subsection{The exponential field $\mathbb T$}
We begin by introducing the multiplicative group~$\mathfrak L$ of {\bf logarithmic monomials}:
this is the  multiplicatively written vector space
\[ \mathfrak{L} = \bigl\{ \ell_0^{\alpha_0} \ell_1^{\alpha_1} \cdots 
\ell_n^{\alpha_n} : 
(\alpha_0, \ldots, \alpha_n) \in \R^{n+1},\ n\in\N \bigr\} \]
over the field $\R$ with basis $(\ell_n)_{n\geq 0}$. We equip $\mathfrak L$ with the unique ordering $\preceq$ making it an ordered vector space over the ordered field $\R$ such that $\ell_n \succ \ell_{n+1}^m$ for each~$m$. We let $\mathbb L:= \R [[ \mathfrak{L} ]]$ be the ordered field of {\bf logarithmic transseries} and let $\exp\colon \mathbb L^{\preceq} \to \mathbb L^{>0,\asymp 1}$ be as defined in \eqref{eq:extension exponential}.
Note that the inverse $\log\colon\mathbb L^{>0,\asymp 1}\to  \mathbb L^{\preceq}$ of $\exp$ extends to a strictly increasing group homomorphism $\log\colon\mathbb L^{>0}\to\mathbb L$ by
\begin{equation}\label{eq:extension log}
\log(\mathfrak m\cdot g):=\log\mathfrak m + \log g\qquad\text{for $\mathfrak m\in\mathfrak L$, $g\in \mathbb L^{>0,\asymp 1}$,}
\end{equation}
where for $\mathfrak m= \ell_0^{\alpha_0} \ell_1^{\alpha_1} \cdots 
\ell_n^{\alpha_n}\in\mathfrak L$, $(\alpha_0, \ldots, \alpha_n) \in \R^n$, we set
$$\log \mathfrak m := \sum_{i=0}^n \alpha_i \ell_{i+1}\in \R[\mathfrak L].$$
So writing $x:=\ell_0$, we have $\ell_1=\log x, \ell_2=\log\log x,\dots$, and in general, $\ell_n$ is the $n$th iterated logarithm of $x$.

\medskip
\noindent
Now let $\mathfrak T_0:=\mathfrak L$, $K_0:=\mathbb L$.
Since there is no reasonable
way to define $\exp f$ as an element of $K_0$ for 
$f\in K_0^\succ$, we enlarge
$K_0=\R[[\mathfrak T_0]]$ to a bigger series field $K_1=\R[[\mathfrak T_1]]$ such that 
$\exp f\in K_1$ for
all $f\in K_0^\succ$: simply take a multiplicative 
copy $\exp(K_0^\succ)$ of the ordered additive
subgroup $K_0^\succ$ of $K_0$, with order-preserving isomorphism 
$$f\mapsto\exp f\colon
K_0^\succ\to\exp(K_0^\succ),$$
and form the direct product of multiplicative groups
$$\mathfrak T_1 := \exp\bigl(K_0^\succ\bigr)\cdot\mathfrak T_0.$$ 
Order $\mathfrak T_1$ lexicographically:
for $f\in K_0^\succ$, $\fm\in\mathfrak T_0$, put
$$\exp(f)\cdot\fm \succcurlyeq 1 \qquad\Longleftrightarrow\qquad f>0 
\text{ or }
(f=0 \text{ and } \fm \succcurlyeq 1 \text{ in $\mathfrak T_0$).}$$ 
The natural identification of $\mathfrak T_0$ with an ordered subgroup of 
$\mathfrak T_1$ makes $K_0$
an ordered subfield of $K_1$. Define
$\exp g \in K_1^{>0}$ for $g\in K_0$ by 
$$
\exp(f+c+\varepsilon):=
\exp(f)\cdot e^{c+\varepsilon}
\qquad (f\in K_0^{\succ},\ c\in\R,\ \varepsilon\in 
K_0^{\prec}),
$$
with $e^{c+\varepsilon}$ as in \eqref{eq:extension exponential}.
Now $K_1$ has the same defect as $K_0$: there is no reasonable way to
define $\exp f$ as an element of $K_1$ for $f\in K_1$ with 
$\fd(f)\succ\mathfrak T_0$. 
In order to add the exponentials of such elements
to $K_1$, enlarge $K_1$ to a field $K_2$ just as 
$K_0$ was enlarged to $K_1$. 
More generally, consider a tuple $(K,A,B,\log)$ where
\begin{enumerate}
\item $K$ is an ordered field;
\item $A$ and $B$ are additive subgroups of $K$ with $K=A\oplus B$ and $B$
convex in $K$;
\item $\log\colon K^{>0}\to K$ is a strictly increasing homomorphism
(the inverse of which we denote by $\exp\colon\log(K^{>0})\to K$)
such that $B\subseteq\log(K^{>0})$.
\end{enumerate}
We call such a quadruple $(K,A,B,\log)$ a {\bf pre-logarithmic ordered field}.
So $(K_0,A_0,B_0,\log_0)$ with 
$A_0:=K_0^\succ$, $B_0:=K_0^\preceq$ and
$\log_0\colon K_0^{>0}\to K_0$ given by \eqref{eq:extension log}
is a pre-logarithmic ordered field. 
Given a pre-logarithmic ordered field $(K,A,B,\log)$, define
a pre-logarithmic ordered field $(K',A',B',\log')$ as follows:
Take a multiplicative copy $\exp(A)$ of the ordered additive
group $A$ with order-preser\-ving isomorphism 
$$\exp_A\colon A\to\exp(A),$$
and put $$K':=K[[\exp(A)]],\quad A':=(K')^\succ,\quad B':=K^\preceq=K\oplus 
(K')^\prec,$$ and 
define $\log'\colon (K')^{>0}\to K'$ by
$$\log'\big(f\cdot\exp_A(a)\cdot(1+\varepsilon)\big) := \log f + a + \log(1+\varepsilon)$$
for $f\in K$, $a\in A$, $\varepsilon\in (K')^\prec$.
Note that $\log'$ extends $\log$ and $K\subseteq\log'\big((K')^{>0}\big)$. 
Moreover, if $K=\R[[\fM]]$ for some multiplicative ordered abelian group
$\fM$, then $K'=\R[[\fM']]$ where $\fM'=\exp(A)\cdot\fM$, ordered
lexicographically.
Inductively, set $$(K_{n+1},A_{n+1},B_{n+1},\log_{n+1}):=
(K_n',A_n',B_n',\log_n').$$ 
Then $K_n=\R[[\mathfrak T_n]]$ with 
$\mathfrak T_n=\mathfrak L\cdot\exp(A_{n-1}\oplus\cdots\oplus A_0)$, and we put
$$\mathbb T := \bigcup_n\, \R[[\mathfrak T_n]], \qquad \mathfrak T:=\bigcup_n\mathfrak T_n.$$
Thus $\mathbb T\subseteq \R[[\mathfrak T]]$ as ordered fields. We let 
$\log\colon\mathbb T^{>0}\to\mathbb T$ be the common extension of all the $\log_n$. The map $\log$ is a strictly increasing group isomorphism $\mathbb T^{>0}\to\mathbb T$, so its inverse $\exp\colon\mathbb T\to\mathbb T^{>0}$ is an exponential function on~$\mathbb T$. It is well-known that the exponential field $\RxLE$ of LE-series embeds into $\mathbb T$ in a natural way. (See~\cite{vdH97}, and also \cite{KT}.) However, this embedding is not onto, since, e.g., the series $\sum_n \frac{1}{\ell_n}=\frac{1}{\ell_0} + \frac{1}{\ell_1} + \cdots$ is an element of $\mathbb L$, but is not an LE-series (since iterated logarithms $\ell_n$ of arbitrary ``depth''~$n$ appear in it).

\section{Character, and an Embedding Result}\label{sec:character etc}

\noindent
In this section we define a cardinal invariant of an ordered set $(M,{<})$, which we call character, and which
is more meaningful than the cardinality of $M$ when one is concerned with realizing cuts in $M$.
We demonstrate this by proving an embedding criterion for models of o-minimal theories (Lemma~\ref{lem:embedd}) used in the proof of Theorem~\ref{thm:main theorem, 2}. We then recall some important examples of o-minimal structures.

\subsection{Character}\label{sec:ctble char}
Let $(M,{<})$ be an ordered set.
Recall that the cofinality of $(M,{<})$, denoted by $\operatorname{cf}(M,{<})$, or $\operatorname{cf}(M)$ for brevity, is the minimal cardinality of a cofinal subset of $M$. 
Recall that  $M$ always has a well-ordered cofinal subset of cardinality $\operatorname{cf}(M)$. 
It is also well-known that if $A$ is an ordered subset of $M$ which is cofinal in~$M$, then $\operatorname{cf}(A)=\operatorname{cf}(M)$.
Dually, the coinitiality of $(M,{<})$, denoted by 
$\operatorname{ci}(M,{<})=\operatorname{ci}(M)$, is the cofinality of $(M,{>})$. 

\medskip
\noindent
A {\bf cut} in $M$ is a subset $C$ of $M$ which is closed downward in $M$, i.e., if $x\in M$ satisfies $x<c$ for some $c\in C$, then $x\in C$. Given a cut $C$ in $M$, an element $x$ in an ordered set extending $M$ is said to {\bf realize} the cut $C$ if $C<x<M\setminus C$.
For every element $x$ in an ordered set extending $M$ with $x\notin M$, $x$ realizes the cut 
$$M^{<x}=\{c\in M: c<x\}$$ 
in $M$, which we call the {\bf cut of $x$ in $M$}.
The {\bf character} of a cut $C$ in $M$ is the pair $\big(\operatorname{cf}(C),\operatorname{ci}(M\setminus C)\big)$.
We define the {\bf character} of the ordered set $(M,{<})$, denoted by $\ch(M,{<})$ or simply by $\ch(M)$, to be the supremum
of $\operatorname{cf}(C)+\operatorname{ci}(M\setminus C)$, as $C$ ranges over all cuts in $M$.
If $A$ is an ordered subset of $M$ then $\ch(A)\leq\ch(M)$.

%Since this also applies to all ordered subsets of~$M$, the hereditary cofinality of $M$ is not larger than the supremum of the cardinalities of all well-ordered subsets of~$M$ (but it may be strictly smaller, as the example where $M$ is a singular cardinal  shows). Similarly,  the {\bf hereditary coinitiality} of $(M,{<})$ is the supremum $\operatorname{hci}(M)$ of all coinitialities of ordered subsets of $M$; this quantity is not larger than the supremum of the cardinalities of all reverse well-ordered subsets of~$M$. If $A$ is an ordered subset of $M$ then clearly $\operatorname{hcf}(A)\leq\operatorname{hcf}(M)$ and $\operatorname{hci}(A)\leq\operatorname{hci}(M)$.

\medskip
\noindent
Every (reverse) well-ordered subset of the ordered set of real numbers is countable, so $\ch(\R)=\aleph_0$.   The following was shown by Esterle~\cite{Esterle} (see also \cite[p.~73]{DMM01}):

\begin{lemma}\label{lem:Esterle}
Let $\mathfrak M$ be an ordered abelian group.
Suppose each \textup{(}reverse\textup{)} well-ordered subset of $\mathfrak M$ is countable. Then each \textup{(}reverse\textup{)} well-ordered subset of $\R[[\mathfrak M]]$ is countable.
\end{lemma}

\noindent
From this we easily obtain:

\begin{cor}\label{cor:Esterle}
Each \textup{(}reverse\textup{)} well-ordered subset of $\mathbb T$ is countable. In particular,
$\ch(\mathbb T)=\aleph_0$.
\end{cor}
\begin{proof}
If the ordered set $M$ is the union of countably many ordered subsets, each of whose (reverse) well-ordered subsets is countable, then each (reverse) well-ordered subset of $M$ is countable. Thus  we only need to show that each (reverse) well-ordered subset of the ordered subfield $\R[[\mathfrak T_n]]$ in the construction of $\mathbb T$ is countable. This follows easily by induction on $n$, using the lemma above.  
\end{proof}

\noindent
Let $\alpha$ be an ordinal and $(M,{<})$ be an ordered set. Then  $(M,{<})$ is called an $\eta_\alpha$-set if for all subsets $A$, $B$ of $M$ with $A<B$ and $\abs{A\cup B}<\aleph_\alpha$, there is an element~$x$ of $M$ with $A<x<B$. Note that if  the ordered set $(M,{<})$  is dense without endpoints, then $(M,{<})$ is an $\eta_\alpha$-set iff $(M,{<})$,
viewed as a structure in the language $\{<\}$,  is $\aleph_\alpha$-saturated (a consequence of  the theory of $(M,{<})$ admitting quantifier elimination). We note the following obvious fact:

\begin{lemma}\label{lem:sat}
Let $N$ be  an $\eta_\alpha$-set extending $M$, where $\aleph_\alpha>\ch(M)$. Then each cut in~$M$ is realized by an element of $N$.
\end{lemma}

%In \cite{TV12}, the ordered set $(M,{<})$ is called $\kappa$-Cantor complete if for each $C,D\subseteq M$ such that $C<D$ and $\operatorname{cf}(C)<\kappa$ and $\operatorname{ci}(D)<\kappa$, there is some $x\in M$ with  $C<x<D$.  Every ordered set extending $M$ which is $\kappa$-Cantor complete, where $\kappa\geq\max(\operatorname{hcf}(M),\operatorname{hci}(M))$, contains a realization for every cut in $M$. 

\subsection{An embedding criterion}
In the following we let $\mathcal L$ be  a first-order language containing a binary relation symbol $<$.
An $\mathcal L$-structure $\mathbf M=(M,{<},\dots)$ expanding a dense linearly ordered set $(M,{<})$ without endpoints is said to be {\bf o-minimal} if each subset of $M$ which is definable in $\mathbf M$ is a finite union of intervals $(a,b)$ (where $a,b\in M_{\pm\infty}$) and singletons $\{c\}$ ($c\in M$). 
(Here and in the rest of this paper, ``definable'' means ``definable, possibly with parameters.'')
By~\cite{KPS}, every structure elementarily equivalent to an o-minimal $\mathcal L$-structure is also o-minimal, and in this case, the complete theory $\operatorname{Th}(\mathbf M)$ of $\mathbf M$ is called o-minimal.
We refer to \cite{vdD93} for a summary of basic facts from the theory of o-minimal structures used in this note.

Let now $T$ be a complete o-minimal $\mathcal L$-theory and let $\mathbf M\models T$.
We recall that the o-minimality assumption implies that given an elementary extension $\mathbf N$ of $\mathbf M$,
there is a one-to-one correspondence between the type of an element $x\in N\setminus M$ over $M$ and the cut of $x$ in $M$. In particular, if $\kappa>\abs{\mathcal L}$, then $\mathbf M$ is $\kappa$-saturated iff its underlying ordered set $(M,{<})$ is $\kappa$-saturated.
We also recall that
given a subset $A$ of $M$, there is a prime model of $T$ over $A$, i.e., a model $\mathbf N$ of $T$ with $A\subseteq N$ such that each elementary map $A\to N'$, where $\mathbf N'\models T$, extends to an elementary embedding $\mathbf N\to\mathbf N'$. This prime model of $T$ over $A$ is unique up to isomorphism over $A$. (See \cite[Theorem~5.1]{PS}.) Given 
an elementary extension $\mathbf N=(N,{<},\dots)$ of $\mathbf M$ and $A\subseteq N$, we let $\mathbf M\langle A\rangle$ denote the prime model of $T$ over $A\cup M$, taken as an elementary substructure of $\mathbf N$ (implicitly understood from context).
We also write $\mathbf M\langle x\rangle$ instead of $\mathbf M\langle \{x\}\rangle$. 
If also $\mathbf M\preceq\mathbf N'$, and $x\in N\setminus M$  and $x'\in N'\setminus M$ realize the same cut in $M$, then there is a (unique)  isomorphism $\mathbf M\langle x\rangle\to\mathbf M\langle x'\rangle$ which is the identity on $M$ and sends $x$ to $x'$.
 
\medskip
\noindent
We can now easily show: 

\begin{lemma}\label{lem:embedd}
Let $\mathbf M=(M,{<},\dots)$ and $\mathbf N=(N,{<},\dots)$ be models of $T$ and suppose that $(N,{<})$ is $\kappa^+$-saturated, where $\kappa = \ch(M)$. 
Then every elementary embedding of an elementary substructure of $\mathbf M$ into $\mathbf N$ extends to an elementary embedding of $\mathbf M$ into $\mathbf N$.
\end{lemma}
\begin{proof}
Let $\mathbf A\preceq\mathbf M$ and let
$h\colon\mathbf A\to\mathbf N$ be an elementary embedding.
By Zorn we may assume that $h$ has no extension to an elementary embedding of an elementary substructure of $\mathbf M$, properly containing $\mathbf A$, into $\mathbf N$. Suppose $M\neq A$ and take $x\in M\setminus A$ arbitrary. By $\kappa^+$-saturation of $(N,{<})$ the image $h(A^{<x})$ of the cut of~$x$ in $A$ is realized in $N$, say by $y\in N$. (Lemma~\ref{lem:sat}.) Then the extension of $h$ to a map $A\cup\{x\}\to N$ with $x\mapsto y$ is elementary, and hence  further extends to an elementary embedding $\mathbf A\langle x\rangle\to\mathbf N$, contradicting the maximality of~$\mathbf A$.
\end{proof}

\noindent
Together with the existence of prime models, this yields an improvement, for o-minimal theories, of the general fact that $\kappa$-saturated models of complete first-order theories are $\kappa$-universal:

\begin{cor}
Let $\mathbf N\models T$ be $\kappa$-saturated, where $\kappa$ is an infinite cardinal. 
Then every model $\mathbf M$ of $T$ with $\ch(M)<\kappa$ elementarily embeds into $\mathbf N$.
\end{cor}

\subsection{Some examples of o-minimal expansions of the ordered field of reals}
All the examples of o-minimal structures that we will need are obtained as expansions of the ordered field 
$\overline{\R}=(\R,{<},{0},{1},{+},{\times})$ of real numbers in  a language extending the
language of ordered rings. (The structure $\overline{\R}$ itself is o-minimal as a consequence of Tarski's quantifier elimination theorem. Its elementary theory is axiomatized by the axioms for real closed ordered fields.)
We remark that to any such o-minimal expansion $\mathbf R$ of $\overline{\R}$, Miller's Growth Dichotomy Theorem \cite{M} applies:
{\it either}\/ $\mathbf R$ is {\bf polynomially bounded} (i.e., for every definable function $f\colon\R\to\R$ there is some $n$ such that $\abs{f(x)}\leq x^n$ for sufficiently large $x$), {\it or}\/ the exponential function $\exp\colon\R\to\R$ is $0$-definable in $\mathbf R$. (For example, $\overline{\R}$ is polynomially bounded.)
One calls $\mathbf R$ {\bf exponentially bounded} if for every definable function $f\colon\R\to\R$ there is some $n$ such that $\abs{f(x)}\leq\exp_n(x)$ for sufficiently large $x$.
(At present, {\it all}\/ known o-minimal expansion of the ordered field of real numbers are exponentially bounded.) Note that polynomial boundedness and exponential boundedness are part of the elementary theory of $\mathbf R$; this allows us to speak of an o-minimal theory extending the theory of real closed ordered fields being
polynomially bounded or exponentially bounded.

\subsubsection{The ordered field of reals with restricted analytic functions}
A restricted analytic function is a function $\R^n\to\R$ which is given on $I^n=[-1,1]^n$ by a power series in $n$ variables with real coefficients converging on a neighborhood of $I^n$, and which is $0$ outside of $I^n$.
The structure $\R_{\operatorname{an}}$ is the expansion of $\overline{\R}$ by the restricted analytic functions $\R^n\to\R$, for varying $n$. 
The structure $\R_{\operatorname{an}}$ was shown to be o-minimal in \cite{vdD86} and model-complete in \cite{DvdD};
$\R_{\operatorname{an}}$ is also polynomially bounded.
A complete axiomatization of $T_{\operatorname{an}}:=\operatorname{Th}(\R_{\operatorname{an}})$ is given in~\cite{DMM94}, where it is also shown that if $\mathfrak M$ is a {\it divisible}\/ ordered abelian group, then the ordered field $\R[[\mathfrak M]]$, expanded to a structure in the language $\mathcal L_{\operatorname{an}}$ of $\R_{\operatorname{an}}$ as indicated in Section~\ref{sec:analytic on Hahn fields} above, is an elementary extension of $\R_{\operatorname{an}}$.
It follows that~$\mathbb T$, viewed as an $\mathcal L_{\operatorname{an}}$-structure in the natural way, is a model of $T_{\operatorname{an}}$ (since it is obtained as the union of the increasing elementary chain of $\mathcal L_{\operatorname{an}}$-structures $\R[[\mathfrak T_n]]$) and hence an elementary extension of~$\R_{\operatorname{an}}$ .

\subsubsection{The exponential field of reals}
The ordered field of real numbers $(\overline{\R},\exp)$ augmented by the real exponential function was shown to be model-complete and o-minimal by Wilkie \cite{W}.

\subsubsection{The exponential field of reals with restricted analytic functions}
This is the expansion $\R_{\operatorname{an},\exp}$ of $\R_{\operatorname{an}}$ by the real exponential function. 
In \cite{vdDM} this structure was shown to be model-complete and o-minimal, by a generalization of Wilkie's proof for $(\overline{\R},\exp)$. A simpler proof, and a complete axiomatization of 
$T_{\operatorname{an,exp}}:=\operatorname{Th}(\R_{\operatorname{an,exp}})$, is given in \cite{DMM94}.
In fact, $T_{\operatorname{an,exp}}$ is axiomatized by $T_{\operatorname{an}}$ together with (the universal closures of) the following statements about $\exp$:
\begin{itemize}
\item[(E1)] $\exp(x+y)=\exp(x)\cdot\exp(y)$;
\item[(E2)] $x<y\rightarrow \exp(x)<\exp(y)$;
\item[(E3)] $x>0 \rightarrow \exists y \exp(y)=x$;
\item[(E4$_n$)] $x>n^2 \rightarrow \exp(x)>x^n$ (for each $n>0$);
\item[(E5)] $-1\leq x\leq 1\rightarrow e(x)=\exp(x)$, where $e$ is the function symbol of $\mathcal L_{\operatorname{an}}$ which represents the restricted analytic function $e\colon\R\to\R$ with $e(x)=e^x$ for $x\in [-1,1]$.
\end{itemize}
Based on this axiomatization, in \cite[Corollary~2.8]{DMM97} it was shown that $\RxLE$ is a model of $T_{\operatorname{an,exp}}$  (and consequently, that $T_{\operatorname{an,exp}}$ is exponentially bounded). We observe here in a similar way:

\begin{lemma}\label{lem:T models Tanexp}
$\mathbb T\models T_{\operatorname{an,exp}}$.
\end{lemma}
\begin{proof}
We already noted that $\mathbb T\models  T_{\operatorname{an}}$.
Clearly the exponential function on $\mathbb T$ satisfies (E1)--(E3) by construction.  

We prove (E4$_n$).  Suppose $n>0$ and $f\in \mathbb T$ is such that $f>n^2$.  We need $\exp(f)>f^n$.  Take $m$ such that $f\in \R[[\mathfrak T_m]]$.  Write $f=g+c+\epsilon$, with $g\in \R[[\mathbb T_m]]^{\succ}$, $c\in \R$, and $\epsilon \in \R[[\mathbb T_m]]^{\prec}$.  We first suppose that $g=0$.  We then have $\fd(\exp(f))=1$ with leading coefficient $e^c$, while $\fd(f^n)=1$ with leading coefficient $c^n$.  Since $c\geq n^2$, we have $e^c> c^n$ and thus $\exp(f)>f^n$.  We now suppose that $g\not=0$.  Then $\fd(\exp(f))=\exp(g)\cdot \fd(\exp(c+\epsilon))\succ \fd(f^n)$ since $\fd(f^n)\in \frak T_m$.  Thus, $\exp(f)>f^n$ if and only if the leading coefficient of $\exp(f)$ is positive; however, this leading coefficient is the coefficient of $\fd(\exp(c+\epsilon))$ in $\exp(c+\epsilon)$, which is positive.

Next we prove (E5).  Suppose that $f\in \mathbb T$ with $-1\leq f\leq 1$.  Take $m$ such that $f\in \R[[\mathfrak T_m]]$.  Then $f\in \R[[\mathfrak T_m]]^{\preceq}$, say $f=c+\epsilon$, with $c\in \R$ and $\epsilon \in \R[[\mathfrak T_m]]^{\prec}$.  Then $e(f)=\exp(f)=e^c\sum_{n=0}^\infty \frac{\epsilon^n}{n!}$.
\end{proof}

\section{$T$-Convexity}\label{sec:T-Convexity}

\noindent
In this section we let ${\mathbf R}=(\R,{<},{0},{1},{+},{\times},\dots)$ be an o-minimal expansion of the ordered field of real numbers in  a language $\cal L$ extending the
language of ordered rings. ``Definable'' will always mean ``definable in $\mathbf R$.'' Note that the $\mathcal L$-reduct 
${}^*\mathbf R$ of ${}^*\R$ is an elementary extension of $\mathbf R$.  We let $T=\operatorname{Th}(\mathbf R)$. Given a $0$-definable set $X\subseteq\R^n$ and $\mathbf S\models T$, we let $X(\mathbf S)$ be the subset of $S^n$ defined by the same $\mathcal L$-formula as $X$ in $\mathbf R$.
We first recall the definition of $T$-convexity from \cite{vdDL} and some fundamental facts concerning
this notion, and then give the proof of the main theorem from the introduction. In the last subsection we give another application of the embedding lemma from Section~\ref{sec:character etc}.

\subsection{Definition and basic properties of $T$-convex subrings}
Let $f\colon X\to\R$ be a $0$-definable function, where $X\subseteq\R^n$. We say that a convex subring $\mathcal O$ of ${}^*\R$ is {\bf closed under $f$} if $f\big(X({}^*\mathbf R)\cap\mathcal O^n\big)\subseteq\mathcal O$.
A convex subring of ${}^*\R$ is called {\bf $T$-convex}
if it is closed under all $0$-definable continuous functions $\R\to \R$. Although this definition only talks about one-variable functions, every $T$-convex subring of ${}^*\R$ is automatically closed under all $0$-definable continuous functions $\R^n\to\R$ for all~$n>0$~\cite[(2.9)]{vdDL}. In fact:

\begin{lemma}\label{lem:T-convex closure properties}
Every $T$-convex subring of ${}^*\R$  is closed under all $0$-definable continuous functions $X\to\R$ where $X\subseteq\R^n$ is open or closed.
\end{lemma}
\begin{proof}
By the definable version of the Tietze Extension Theorem (see \cite[Lemma~6.6]{AF} or \cite[Chapter~8]{vdD98}), each $0$-definable continuous function $X\to\R$ on a closed set $X\subseteq\R^n$ has an extension to a $0$-definable continuous function $\R^n\to\R$. This implies the lemma for $0$-definable continuous functions whose domain is a closed subset of $\R^n$, and this in turn yields the lemma also for $0$-definable continuous functions whose domain is open: if $X\subseteq\R^n$, $X\neq\R^n$, is open then
$X = \bigcup_{s>0} X_s$ where
$$X_s:= \big\{ x\in\R^n: \operatorname{dist}(x,\R^n\setminus X)\geq s\big\}$$
is closed, for $s\in\R^{>0}$.
\end{proof}

\begin{examples}
If $T$ is polynomially bounded then every convex subring of ${}^*\R$ is $T$-convex. If $T$ is non-polynomially bounded but exponentially bounded, then a convex subring $\mathcal O$ of ${}^*\R$ is $T$-convex iff $\mathcal O$ is closed under $\exp\colon\R\to\R$.  
In particular, the convex subring $\mathcal E$ of ${}^*\R$, defined in the introduction, is $T_{\operatorname{an},\exp}$-convex, since it is closed under $\exp$ (and is, indeed, the smallest convex subring of ${}^*\R$ containing $\R$ and the infinite element $\xi$ which is closed under $\exp$).
\end{examples}

\noindent
It is an easy consequence of the Monotonicity Theorem for o-minimal structures that the convex hull in ${}^*\R$ of an elementary substructure of ${}^*\mathbf R$ is a $T$-convex subring of ${}^*\mathbf R$.
A strong converse of this observation is shown in \cite{vdDL}, stated in the theorem below. Note that if $\mathcal O$ is a convex subring of ${}^*\mathbb R$ and $R'$ is a subring of $\mathcal O$ which is a field, then the composition $R'\to\mathcal O\to\widehat{\mathcal O}$ is injective, and this map is bijective iff $\mathcal O=\mathbf R'+\mathcal o$; in this case
$\mathcal o = \{x\in {}^*\R: \abs{x}< (R')^{>0}\}$, so $\mathcal O$ is the convex hull of $R'$ in ${}^*\R$. 

\begin{theorem}[van den Dries-Lewenberg \cite{vdDL}]
Let $\mathcal O$ be a $T$-convex subring of ${}^*\mathbf R$. Then:
\begin{enumerate}
\item an elementary substructure $\mathbf R'$ of ${}^*\mathbf R$ contained in $\mathcal O$ is maximal among the elementary substructures of ${}^*\mathbf R$ contained in $\mathcal O$ iff $\mathcal O=\mathbf R'+\mathcal o$ \textup{(}and hence, by Zorn, there is some $\mathbf R'\preceq{}^*\mathbf R$ with $\mathcal O=\mathbf R'+\mathcal o$\textup{)};
\item if $\mathbf R'$ and $\mathbf R''$ are both maximal with respect to being elementary substructures of ${}^*\mathbf R$ contained in $\mathcal O$, then there is a unique isomorphism $h\colon \mathbf R'\to\mathbf R''$ such that $\widehat{h(x)}=\widehat{x}$ for all $x\in R'$.
\end{enumerate}
\end{theorem}

\noindent
In the following we fix a $T$-convex subring $\mathcal O$ of~${}^*\mathbf R$. By part (1) of the theorem above, we can make $\widehat{\mathcal O}$ into a model of $T$ as follows: take an $\mathbf R'$ which is maximal among elementary substructures of ${}^*\mathbf R$ contained in $\mathcal O$, and make $\widehat{\mathcal O}$ into an $\mathcal L$-structure such that the restriction of the residue map $x\mapsto\widehat{x}\colon \mathcal O\to\widehat{\mathcal O}$ is an isomorphism $\mathbf R'\to\widehat{\mathcal O}$. By part (2) of the above theorem, this expansion of $\widehat{\mathcal O}$ to a model of $T$ is independent of the choice of $\mathbf R'$. In the following we will always consider 
$\widehat{\mathcal O}$ as a model of $T$ in this way.

\medskip
\noindent
The following is  \cite[(2.20)]{vdDL}:

\begin{lemma}\label{lem:induced}
Let $f\colon \R^n\to\R$ be a $0$-definable continuous function. Then for all $x,y\in \mathcal O^n$  we have:
$x-y\in\mathcal o^n \ \Longrightarrow\ f(x)-f(y)\in\mathcal o$.
\end{lemma}

\noindent
In particular, if $f$ is as in the previous lemma, and $\widehat f$ denotes the function $\widehat{\mathcal O}^n\to\widehat{\mathcal O}$ defined by the same formula in the $\mathcal L$-structure $\widehat{\mathcal O}$ as $f$ in $\mathbf R$, then by the lemma, we have $\widehat{f}(\widehat{x})=\widehat{f(x)}$ for all $x\in\mathcal O^n$. Here and below, for $n>0$ and $x=(x_1,\dots,x_n)\in\mathcal O^n$ 
we  set $\widehat{x}=(\widehat{x_1},\dots,\widehat{x_n})\in\widehat{\mathcal O}^n$.

\medskip
\noindent
It was noted in \cite{TV12,TV09} that  (as a consequence of the Mean Value Theorem) the exponential function on ${}^*\R$ induces a function on $\widehat{\mathcal E}$; that is, 
there is a function $\widehat{\exp}\colon \widehat{\mathcal E}\to \widehat{\mathcal E}$  such that
$\widehat{\exp}(\widehat{x})=\widehat{\exp(x)}$ for all $x\in\mathcal E$. 
By the discussion above, applied to $\mathbf R=(\overline{\R},\exp)$, we now see that 
this function $\widehat\exp$ agrees with the interpretation of the function symbol $\exp$ of the language $\mathcal L_{\exp}$ of $(\overline{\R},\exp)$ in the $\mathcal L_{\exp}$-structure $\widehat{\mathcal E}$
(which is a model of $T_{\exp}$). For $T=T_{\operatorname{an}}$, the interpretations of the function symbols for restricted analytic functions in the $\mathcal L_{\operatorname{an}}$-structure $\widehat{\mathcal O}$ are induced, in a similar way, by the interpretations of those symbols in ${}^*\mathbf R$. This is an immediate consequence of a variant of Lemma~\ref{lem:induced}, which we formulate and prove below.

For this, suppose that $f\colon X\to \R$ is a $0$-definable continuous function, where $X\subseteq \R^n$ is open.    Let $\hat{f}\colon X(\hat{\mathcal O})\to \hat{\mathcal O}$ be the function defined by the same formula that defines $f$ in $\mathbf R$.  Then for $a\in X(\mathbf R')$ we have $\hat{f}(\hat{a})=\hat{f(a)}$.
Recall that by Lemma~\ref{lem:T-convex closure properties}, $\mathcal O$ is closed under $f$.

\begin{lemma}\label{lem:induced, 2}
Let $a\in X(\mathbf R')$ and $b\in \mathcal O^n$ such that $a-b\in \mathcal o^n$. 
Then $b\in X({}^*\mathbf R)$ and $f(a)-f(b)\in \mathcal o$.  
\end{lemma}

\begin{proof}
Take $\delta\in (R')^{>0}$ such that $\mathbf{R'}\models \forall x(|x-a|<\delta \rightarrow x\in X)$.
Since  $\mathbf R'$ is an elementary substructure of ${}^*\mathbf R$ and $|a-b|<\delta$, we have
$b\in X({}^*\mathbf R)$. Now
fix $\epsilon \in (R')^{>0}$.  We need to prove that $|f(a)-f(b)|<\epsilon$.  Take $\delta \in (R')^{>0}$ such that $\mathbf{R'}\models \forall x(|x-a|<\delta \rightarrow |f(x)-f(a)|<\epsilon)$.  Since $\mathbf R'\preceq {}^*\mathbf R$ and $|a-b|<\delta$, we obtain $|f(a)-f(b)|<\varepsilon$.
\end{proof}

\noindent
Now suppose that $b\in \mathcal O^n$ is such that $\hat{b}\in X(\hat{\mathcal O})$.  Take $a\in (R')^n$ such that $\hat{a}=\hat{b}$.  Then $a\in X(\mathbf R')$, and $b\in X({}^*\mathbf R)$ and $f(a)-f(b)\in \mathcal o$ by the previous lemma.  Consequently, we see that $\hat{f}(\hat{b})=\hat{f}(\hat{a})=\hat{f(a)}=\hat{f(b)}$.

\medskip
\noindent 
In Section~\ref{sec:differentiability in residue fields} below we show that not only each restricted analytic function on $\R^n$, but {\it each}\/ restricted $C^\infty$-function on $\R^n$ induces a function on $\widehat{\mathcal O}^n$. However, in general, the connection between this induced function and its original seems less tight. (See the question at the end of Section~\ref{sec:differentiability in residue fields}.)

\subsection{Proof of the main theorem}
We are now ready to give a proof of the main theorem from the introduction.
We first observe:

\begin{lemma}\label{lem:embedd, 2}
Let $\mathbf S$ be an elementary extension of $\mathbf R$, and let $\mathbf S'$ be an elementary extension of $\mathbf S$ such that $S$ is cofinal in $S'$. Let $h\colon \mathbf S\to{}^*\mathbf R$ be an elementary embedding such that $h(S)\subseteq\mathcal O$. Suppose the ordered set ${}^*\R$ is $\kappa^+$-saturated where $\kappa=\operatorname{ch}(S')$. Then $h$ extends to an elementary embedding $\mathbf S'\to {}^*\mathbf R$ whose image is contained in $\mathcal O$.
\end{lemma}
\begin{proof}
By Lemma~\ref{lem:embedd}, $h$ extends to an elementary embedding $\mathbf S'\to {}^*\mathbf R$, also denoted by~$h$. Since $S$ is cofinal in $S'$ and $h(S)\subseteq\mathcal O$, we have $h(S')\subseteq\mathcal O$.
\end{proof}

\noindent
In connection with the following proposition we note that \cite[(2.13)]{vdDL} shows that if  $\xi\in\mathcal O$ with $\xi>\R$, then $\mathbf R\langle\xi\rangle\subseteq\mathcal O$.

\begin{prop}\label{prop:embedd}
Let $\mathbf S$ be an elementary extension of $\mathbf R$ and $x\in S$, $x>\R$, such that $\mathbf R\langle x\rangle$ is cofinal in $S$.
Suppose the ordered set ${}^*\R$ is $\kappa^+$-saturated, where $\kappa=\ch(S)$. 
Then for each $\xi\in\mathcal O$, $\xi>\R$, there is an elementary embedding $\mathbf S\to {}^*\mathbf R$
which is the identity on~$\R$ and sends $x$ to $\xi$, and whose image is contained in $\mathcal O$
\textup{(}and hence there is an elementary embedding $\mathbf S\to\widehat{\mathcal O}$ which is the identity on $\R$ and sends $x$ to $\widehat{\xi}$\textup{)}.
\end{prop}
\begin{proof}
Let $\xi\in\mathcal O$ with $\xi>\R$.
By o-minimality, take an isomorphism $h\colon\mathbf R\langle x\rangle \to \mathbf R\langle\xi\rangle$ with $h(r)=r$ for all $r\in\R$ and $h(x)=\xi$. By the remark preceding the proposition and the previous lemma, $h$ extends to an elementary embedding $\mathbf S\to {}^*\mathbf R$, also denoted by~$h$, with $h(S)\subseteq\mathcal O$.
By Zorn, take an elementary substructure $\mathbf R'$ of ${}^*\mathbf R$, maximal subject to the conditions $h(S)\subseteq  R'\subseteq\mathcal O$.
The residue map $\mathcal O\to\widehat{\mathcal O}$ now restricts to an isomorphism
$\mathbf R'\to\widehat{\mathcal O}$ of $\mathcal L$-structures, and pre-composition with $h$ yields the desired elementary embedding $\mathbf S\to\widehat{\mathcal O}$.
\end{proof}

\noindent
The previous proposition immediately yields a more precise form of Theorem~\ref{thm:main theorem} from the introduction:

\begin{theorem}\label{thm:main theorem, 2}
Suppose the ordered set ${}^*\R$ is $\aleph_1$-saturated and let $\mathcal O$ be a convex subring of ${}^*\R$ which is closed under $\exp$.
Then for each $\xi>\R$ in $\mathcal O$ there exists an  
embedding $\mathbb T\to\widehat{\mathcal O}$ of $\mathcal L_{\operatorname{an},\exp}$-structures 
over $\R$ with $x\mapsto\xi$.
\end{theorem}
\begin{proof}
We apply the above material to $\mathbf R=\R_{\operatorname{an},\exp}$ (so 
$T=T_{\operatorname{an},\exp}$). By the remark following Lemma~\ref{lem:T-convex closure properties}, each convex subring of ${}^*\R$ which is closed under $\exp$ is a $T$-convex subring of ${}^*\mathbf R$. 
By Corollary~\ref{cor:Esterle}, the ordered set $\mathbb T$ has countable character.
Moreover, by construction of $\mathbb T$, the sequence $(\exp_n(x))$ of iterated exponentials of~$x$ is cofinal in $\mathbb T$.
Hence the theorem follows from the previous proposition.
\end{proof}

\noindent
By Lemmas~\ref{lem:Esterle} and \ref{lem:embedd, 2} and (the proof of) the previous theorem, we obtain:

\begin{cor}\label{cor:main theorem}
Suppose the ordered set ${}^*\R$ is $\aleph_1$-saturated and let $\mathcal O$ be a convex subring of ${}^*\R$ which is closed under $\exp$. Then there is an embedding of $\mathcal L_{\operatorname{an}}$-structures
$\R[[\mathfrak T]]\to\widehat{\mathcal O}$ over $\R$ which restricts to an embedding 
$\mathbb T\to\widehat{\mathcal O}$ of $\mathcal L_{\operatorname{an},\exp}$-structures 
with $x\mapsto\xi$.
\end{cor}

\noindent
We finish this section with another application of the embedding lemma from Section~\ref{sec:character etc}, to Conway's surreal numbers. 

\subsection{Embedding $\mathbb T$ into the surreals}
The surreal numbers form a (proper) class $\text{\bf No}$  equipped with a linear ordering, extending the ordered class
of all ordinal numbers, and also
containing $\R$ as an ordered subset in a natural way.
This ordered class comes with natural algebraic operations  making it a real closed ordered field extension of $\overline{\R}$. The remarkable characteristic property of the ordered field $\text{\bf No}$ is that it is the {\it homogeneous universal}\/ ordered field: every ordered field whose universe is a {\it set}\/ embeds into  $\text{\bf No}$, and any isomorphism between subfields of
$\text{\bf No}$ whose universes are sets extends to an automorphism of $\text{\bf No}$.
We refer to \cite{Gonshor} for the construction and basic properties of the class $\text{\bf No}$.
We recall in particular that each surreal number has a {\it length,}\/ which is an ordinal, and
that the collection $\text{\bf No}(\lambda)$ of surreal numbers of length less than a given ordinal $\lambda$
forms a set, with  $\mathbb R\subseteq \text{\bf No}(\omega+1)$. Also, the ordinal $\omega$, viewed as a surreal, is larger than every real number, viewed as a surreal: $\omega>\R$.

Recall that an $\varepsilon$-number is an ordinal $\lambda$ with the property that $\omega^\lambda=\lambda$.
For example, every uncountable cardinal is an $\varepsilon$-number. The smallest $\varepsilon$-number is
$\varepsilon_0=\sup\{\omega,\omega^\omega,\omega^{\omega^\omega},\dots\}$.
If $\lambda$ is an $\varepsilon$-number, then $\text{\bf No}(\lambda)$ is a subfield of $\text{\bf No}$.
It was already noted by Kruskal (see \cite[Chapter~10]{Gonshor}) that the exponential function on $\R$
extends to an exponential function on the ordered field $\text{\bf No}$. 
In \cite{vdDE}, van den Dries and Ehrlich show that if $\lambda$ is an $\varepsilon$-number then
$\text{\bf No}(\lambda)$ is closed under exponentiation and under taking logarithms of positive elements, and 
the exponential field $(\text{\bf No}(\lambda),\exp)$
is an elementary extension of $(\overline{\mathbb R},\exp)$. In fact, \cite{vdDE} also shows that in this case, the exponential field $(\text{\bf No}(\lambda),\exp)$ can be expanded to an $\mathcal L_{\operatorname{an},\exp}$-structure
making $\text{\bf No}(\lambda)$ an elementary extension of $\mathbb R_{\operatorname{an},\exp}$.
We now obtain a complement to \cite[Theorem~19]{Ehrlich12}:

\begin{cor}\label{cor:surreal}
There is an embedding of $\mathcal L_{\operatorname{an}}$-structures
$\R[[\mathfrak T]]\to \text{\bf No}(\omega_1)$ over $\R$ which restricts to an embedding 
$\mathbb T\to\text{\bf No}(\omega_1)$ of $\mathcal L_{\operatorname{an},\exp}$-structures 
with $x\mapsto\omega$.
\end{cor}
\begin{proof}
Let $\R\langle x\rangle$ denote the elementary $\mathcal L_{\operatorname{an},\exp}$-substructure of $\mathbb T$ 
generated by $x$ over $\R$, and similarly let $\R\langle\omega\rangle$ denote the  elementary $\mathcal L_{\operatorname{an},\exp}$-substructure of $\text{\bf No}(\omega_1)$ 
generated by $\omega$ over $\R$. (Note that $\R\langle x\rangle\subseteq\RxLE$ and $\R\langle \omega\rangle\subseteq \text{\bf No}(\varepsilon_0)$.)
By o-minimality, we have an isomorphism $h\colon \R\langle x\rangle \to \R\langle\omega\rangle$ of $\mathcal L_{\operatorname{an},\exp}$-structures with $h(r)=r$ for each $r\in\R$ and $h(x)=\omega$.
The underlying ordered set of $\text{\bf No}(\omega_1)$ is $\aleph_1$-saturated (see \cite[Lemma~1]{Ehrlich88} and 
proof of \cite[Theorem~17]{Ehrlich12}). The claim now follows from Lemmas~\ref{lem:Esterle} and \ref{lem:embedd}.
\end{proof}

\noindent
This corollary leads to a number of natural questions which we do not pursue here, e.g.: what is the smallest $\epsilon$-number $\lambda$ such that there is an embedding 
$\mathbb T\to\text{\bf No}(\lambda)$ of $\mathcal L_{\operatorname{an},\exp}$-structures 
with $x\mapsto\omega$?

\section{Differentiability in Residue Fields, and a Question}\label{sec:differentiability in residue fields}

\noindent
As in the introduction we let $\mathcal O$ be a convex subring of ${}^*\R$, with maximal ideal~$\mathcal o$, residue field $\widehat{\mathcal O}=\mathcal O/\mathcal o$, and residue morphism $x\mapsto\widehat{x}\colon\mathcal O\to\widehat{\mathcal O}$. In this last section of the paper we show that each restricted $C^\infty$-function $\R^n\to\R$ extends to a restricted  $C^\infty$-function $\widehat{\mathcal O}^n\to\widehat{\mathcal O}$ in a natural way.

\begin{notation}
For $X\subseteq\R^n$ we let
$$\mu(X) := \big\{ x\in ({}^*\R_{\operatorname{fin}})^n: \st(x)\in X\bigr\}$$
be the monad of $X$, and we let
$$\widehat\mu(X) := \big\{ \widehat{x}: x\in\mu(X)\big\}\subseteq\widehat{\mathcal O}^n$$
be the image of $\mu(X)$ under $x\mapsto\widehat{x}$. 
So for example, 
let  
\begin{equation}\label{eq:I}
B=(a_1,b_1)\times\cdots\times (a_n,b_n)\quad\text{ where $a_i,b_i\in\R$ with $a_i<b_i$}
\end{equation}
be an open box in $\R$; then
$$\mu(B) = (a_1,b_1)_{{}^*\R_{\operatorname{fin}}}\times\cdots\times (a_n,b_n)_{{}^*\R_{\operatorname{fin}}}, \quad
\widehat\mu(B) = [a_1,b_1]_{\widehat{\mathcal O}}\times\cdots\times [a_n,b_n]_{\widehat{\mathcal O}}.$$
\end{notation}

\noindent
We first show that each infinitely differentiable real-valued function defined on an open set $X\subseteq\R^n$ extends in a natural way to an infinitely differentiable function on the interior 
$\operatorname{int}\widehat{\mu}(X)$
of $\widehat{\mu}(X)$ (taking values in $\widehat{\mathcal O}$). Given $x,y\in {}^*\R^n$ we denote by 
$$[x,y]:=\big\{t x+(1-t)y:t\in [0,1]_{{}^*\R}\big\}$$ 
the line segment in ${}^*\R$ between $x$ and $y$.

\medskip
\noindent
Let first $f\colon B\to\R$ be $C^1$, where $B$ is as in \eqref{eq:I}.

\begin{lemma}
Let $x\in \mu(B)$. Then $f(x)\in{}^*\R_{\operatorname{fin}}$, and if in addition $y\in\mu(B)$ and
 $x-y\in \mathcal o^n$, then $f(x)-f(y)\in\mathcal o$.
\end{lemma}
\begin{proof}
We have $f(x)\approx f(\st(x))$, hence $f(x)\in{}^*\R_{\operatorname{fin}}$.
Now let $y\in\mu(B)$ and $x-y\in \mathcal o^n$. By the Mean Value Theorem, there is $z\in [x,y]$ such that $f(y)-f(x)=f'(z)\cdot (y-x)$.
However, $\st(z)=\st(x)\in B$, and $f'$ is continuous, so $f'(z)\approx f'(\st(z))$ and hence
$f(y)-f(x)=f'(z)\cdot(y-x)$ where $f'(z)\in ({}^*\R_{\operatorname{fin}})^n$ and $y-x\in \mathcal o^n$. Thus  $f(y)-f(x)\in\mathcal o$ as required.
\end{proof}

\noindent
The previous lemma allows us to define a function 
$$\widehat{f}\colon \widehat{\mu}(B)\to\widehat{\mathcal O},\qquad \widehat{f}(\widehat{x}):=\widehat{f(x)}\quad\text{ for each $x\in \mu(B)$.}$$
Next we show:

\begin{lemma}\label{lem:hat f, 1}
The function $\widehat{f}$ is continuous.
\end{lemma}
\begin{proof}
Fix $x\in\mu(B)$. Let $\varepsilon\in\mathcal O$ with $\varepsilon>\mathcal o$. We need to find $\delta\in\mathcal O$ with $\delta>\mathcal o$ such that for all $y\in\mu(B)$ such that $\abs{x-y}<\delta$ we have $\abs{f(x)-f(y)}<\varepsilon$.
For this, we may assume that $\varepsilon\in{}^*\R_{\operatorname{inf}}$. Take a closed box $C\subseteq B$ whose interior contains $\st(x)$,  and $M\in\N^{>0}$  such that $\abs{f'(y)}\leq M$ for all $y\in B$. We claim that $\delta=\frac{\varepsilon}{M}$ works. Indeed, let 
$y\in\mu(B)$ satisfy $\abs{x-y}<\delta$. Take $z\in [x,y]$ such that $f(y)-f(x)=f'(z)\cdot (y-x)$. Then $y-x\in ({}^*\R_{\operatorname{inf}})^n$ and thus $\st(z)=\st(x)$, so $z\in {}^*C$ and $\abs{f'(z)}\leq M$. Hence  $\abs{f(y)-f(x)}\leq\abs{f'(z)}\cdot \abs{y-x}<M\cdot\delta=\varepsilon$.
\end{proof}

\begin{lemma}\label{lem:hat f, 2}
Suppose $f$ is $C^2$. Then the restriction of $\widehat{f}$ to $\operatorname{int}\widehat{\mu}(B)$ is differentiable with derivative $\widehat{f'}$.
\end{lemma}
\begin{proof}
Fix $x\in \mu(B)$. Let $e_1,\dots,e_n\in\R^n$ be the standard basis vectors of $\R$, and fix $i\in\{1,\dots,n\}$.
Let $\varepsilon\in{}^*\R_{\operatorname{inf}}$ with $\varepsilon>\mathcal o$; we need to find $\delta\in\mathcal O$ with $\delta>\mathcal o$ such that for each $h\in{}^*\R$ with $0<\abs{h}<\delta$ we have
$$\left|\frac{f(x+he_i)-f(x)}{h}-\frac{\partial f}{\partial x_i}(x)\right| < \varepsilon.$$
Again, let $C$ be a closed box contained in $B$ such that $\operatorname{st}(x)$ is in the interior of~$C$. Let $M\in\N^{>0}$ be such that $\abs{(\frac{\partial f}{\partial x_i})'(y)}\leq M$ for all $y\in C$. We claim that $\delta=\frac{\varepsilon}{M}$ works. Suppose  $h\in{}^*\R$, $0<\abs{h}<\delta$. Take $y\in [x,x+h]$ such that $\frac{f(x+he_i)-f(x)}{h}=\frac{\partial f}{\partial x_i}(y)$. We need $\abs{\frac{\partial f}{\partial x_i}(y)-\frac{\partial f}{\partial x_i}(x)}<\varepsilon$. Take $z\in [x,y]$ such that $\frac{\partial f}{\partial x_i}(y)-\frac{\partial f}{\partial x_i}(x)=(\frac{\partial f}{\partial x_i})'(z)\cdot (y-x)$. Since $\st(z)=\st(x)$ we have $z\in {}^*C$ and thus $\abs{(\frac{\partial f}{\partial x_i})'(z)}\cdot\abs{y-x}<M\cdot\delta=\varepsilon$, as required.
\end{proof}

\begin{cor}\label{cor:hat f}
Let $N\in\N$.
If $f$ is $C^{N+1}$, then $\widehat{f}\upharpoonright\operatorname{int}\widehat{\mu}(B)$ is $C^N$. In particular, if $f$ is~$C^\infty$, then $\widehat{f}\upharpoonright\operatorname{int}\widehat{\mu}(B)$ is $C^\infty$.
\end{cor}
\begin{proof}
By induction on $N$, where the case $N=0$ holds by Lemma~\ref{lem:hat f, 1}. Now suppose that $f$ is $C^{N+2}$. Then $f'\colon B\to\R^n$ is $C^{N+1}$, so $\widehat{f'}\upharpoonright\operatorname{int}\widehat{\mu}(B)$ is $C^N$ by induction. But $\widehat{f'}=(\widehat{f})'$ on $\operatorname{int}\widehat{\mu}(B)$, by Lemma~\ref{lem:hat f, 2}, so $\widehat{f}\upharpoonright\operatorname{int}\widehat{\mu}(B)$ is $C^{N+1}$.
\end{proof}

\noindent
Let $X$ be an arbitrary open subset of $\R^n$ and let $f\colon X\to\R$ be $C^1$. Then $X$ is a union of open boxes $B$ of the form \eqref{eq:I}, and $\widehat{\mu}(X)$ is then the union of the corresponding sets $\widehat{\mu}(B)$. Hence 
by the above, we can define a function 
$$\widehat{f}\colon \widehat{\mu}(X)\to\widehat{\mathcal O},\qquad \widehat{f}(\widehat{x}):=\widehat{f(x)}\quad\text{ for each $x\in \mu(X)$,}$$
and then the analogues of Lemmas~\ref{lem:hat f, 1} and \ref{lem:hat f, 2}, and Corollary~\ref{cor:hat f} hold (with $X$ in place of $B$).

\medskip
\noindent
Let now $I:=[-1,1]$, and for an ordered field extension $K$ of $\R$ write 
$$I(K):=\{x\in K:-1\leq x\leq 1\}.$$ 
Let
$\mathcal C^\infty$ be the collection of all functions $I^n\to\R$, for varying $n$, which extend to a $C^\infty$-function on a neighborhood of $I^n$, and let $\mathcal C\subseteq\mathcal C^\infty$.
A {\bf restricted $\mathcal C$-function} is a function $\R^n\to\R$ which on $I^n$ agrees with a function in $\mathcal C$ and which is $0$ on $\R^n\setminus I^n$.
Let $\mathcal L_{\mathcal C}$ be the language of ordered rings augmented by an $n$-ary function symbol for each restricted $\mathcal C$-function $\R^n\to\R$; we use the same letter to denote the function from $\mathcal C$ and its corresponding function symbol.
We may expand the ordered field~$\widehat{\mathcal O}$ to an $\mathcal L_{\mathcal C}$-structure by interpreting each function symbol $f$, where  
$f\colon \R^n\to\R$ is a restricted $\mathcal C$-function, by the function $f^{\widehat{\mathcal O}}\colon \widehat{\mathcal O}^n\to\widehat{\mathcal O}$ given by
$$f^{\widehat{\mathcal O}}(\widehat{x}) = 
\begin{cases}
\widehat{f}(\widehat{x}) 	& \text{if $\widehat{x}\in I(\widehat{\mathcal O})^n$,} \\
0							& \text{otherwise.}
\end{cases}
$$
Note that this definition makes sense by the discussion above, since $f$ agrees with a $C^\infty$-function on an open neighborhood of $I^n$.
Also note that $\R$ is the underlying set of a substructure $\R_{\mathcal C}$ of the $\mathcal L_{\mathcal C}$-structure~$\widehat{\mathcal O}$.

\begin{question}
Are there natural conditions on $\mathcal C$ which ensure that the $\mathcal L_{\mathcal C}$-structure $\R_{\mathcal C}$ is an elementary substructure of $\widehat{\mathcal O}$?
\end{question}

\noindent
One such natural condition is that $\R_{\mathcal C}$ be o-minimal: in this case, $\widehat{\mathcal O}$
can be made into an elementary extension of $\R_{\mathcal C}$ as explained in the previous section, 
and $f^{\widehat{\mathcal O}}$ then agrees with the interpretation of the function symbol $f$ in this 
$\mathcal L_{\mathcal C}$-structure, by Lemma~\ref{lem:induced, 2} and the discussion surrounding it.
Example of families $\mathcal C$ which make $\R_{\mathcal C}$ o-minimal are, of course, the family consisting of all  restrictions to the unit cubes $I^n$ of analytic functions on neighborhorhoods of $I^n$ (so $\R_{\mathcal C}=\R_{\operatorname{an}}$), or the family of  all  restrictions to the unit cubes of $C^\infty$-functions associated to a given Denjoy-Carleman class~\cite{RSW}. Also, Le Gal~\cite{Le Gal} and Grigoriev~\cite{Grigoriev}  have shown that the expansion of the real field by a {\it generic}\/ restricted $C^\infty$-function is o-minimal.

\medskip
\noindent
In general, however, $\R_{\mathcal C}$ is not an elementary substructure of $\widehat{\mathcal O}$, as was pointed out to us by Dave Marker:

\begin{example}\label{lem:dave}
Suppose that $\N$ is definable in $\R_{\mathcal C}$, say by
the $\mathcal L_{\mathcal C}$-formula  $\varphi(x)$, possibly involving parameters.
(This hypothesis is satisfied if $\mathcal C$ contains a restricted $C^\infty$-function with an infinite discrete zero set which is in semialgebraic bijection with~$\N$, such as the restricted $C^\infty$-function $s$ with $s(0)=0$ and
$s(x) = e^{-1/x^2}\sin(1/x)$ for $x\in [-1,1]\setminus \{0\}$.)
Suppose further that $\mathcal O$ is the convex hull of $\R[\xi]$ in ${}^*\R$, where $\xi\in {}^*\R$ is positive infinite.  We then claim that $\R_{\mathcal C}$ is not an elementary substructure of $\widehat{\mathcal O}$.  Indeed, suppose that $\widehat{\mathcal O}\models \varphi(\alpha)\wedge\alpha>\xi$.  Let $\psi(x,y)$ define the graph of the function $x\mapsto 2^x\colon\mathbb{N}\to \mathbb{N}$ in $\mathbb R_{\mathcal C}$.  Then if $\widehat{\mathcal O}\models \psi(\alpha,\beta)$, then by elementarity, we have $\widehat{\mathcal O}\models \beta>\hat{\xi}^n$, for each $n$, contradicting the fact that $\hat \xi, \hat \xi^2,\ldots$ is cofinal in $\widehat{\mathcal O}$. 
\end{example}

\noindent
Of course, once we can define the set $\N$  in $\R_{\mathcal C}$,   we immediately define all projective sets; in particular, we define all $C^\infty$-functions, making the preceding example a very wild one.
On the other hand, even in this situation, we sometimes obtain an {\it existentially closed}\/ substructure:

\begin{lemma}\label{lem:tressl}
Suppose ${}^*\R$ is $\mathfrak c^+$-saturated. Then
there is a convex subring $\mathcal O$ of ${}^*\R$ with $\R_{\operatorname{fin}} \subsetneq \mathcal O\subsetneq {}^*\R$, such that $\R_{\mathcal C^\infty}\preceq_1 \widehat{\mathcal O}$.
\end{lemma}

\noindent
This follows from results in \cite{Tressl}; we recall some definitions and basic results from that paper.

\medskip
\noindent
Let $\mathcal L_{\operatorname{src}}$ be the expansion of the language of rings by an $n$-ary function
symbol~$f$ for every continuous function $f\colon\R^n\to\R$ (for varying $n$).
A {\bf super real closed ring} is an $\mathcal L_{\operatorname{src}}$-structure which expands a commutative ring with $1$, with the interpretations of $+$, $\cdot$ and $0$, $1$ compatible with the interpretations of the functions symbols associated to the corresponding continuous functions $\R^2\to\R$ and $\R^0\to\R$, respectively,
and satisfying the $\mathcal L_{\operatorname{src}}$-sentences
$$\forall x(\id(x)=x) \qquad\text{(where $\id\colon\R\to\R$ is the identity function)}$$ 
and
$$\forall x_1 \cdots \forall x_n \bigg( f\big(g_1(x_1),\dots,g_n(x_n)\big) = \big(f\circ (g_1,\dots,g_n)\big)(x_1,\dots,x_n) \bigg),$$
for all continuous functions $f\colon\R^n\to\R$ and $g_i\colon \R^{m_i}\to\R$ ($i=1,\dots,n$).  The class of super real closed rings is a variety (in the sense of universal algebra). 

Clearly the field $\R$ can be expanded to an $\mathcal L_{\operatorname{src}}$-structure in a natural way.
Let~$A$ be a super real closed ring. Then $A$ can be viewed as an extension of  $\R$ in a unique way (identifying each $r\in\R$ with the interpretation of the constant function $r$ in $A$), and if $A$ is an integral domain, then $\R$ is existentially closed in $A$ \cite[Corollary~5.6]{Tressl}.
A prime ideal $P$ of $A$ is said to be {\bf $\Upsilon$-radical} if it is the kernel of a morphism $A\to B$ of
$\mathcal L_{\operatorname{src}}$-structures, for some super real closed ring~$B$.
(See \cite[Sections~3 and 6]{Tressl} for an explanation of this terminology.)
In this case, there is a unique expansion of the integral domain $A/P$ to a super real closed ring such that the residue map $A\to A/P$ is a morphism of $\mathcal L_{\operatorname{src}}$-structures. Every maximal ideal of $A$ is $\Upsilon$-radical \cite[Theorem~6.14]{Tressl}. Given a subset $C$ of $A$, there exists a smallest super real closed subring of $A$ containing~$C$;  the cardinality of this super real closed subring of $A$ equals $\abs{C}+\mathfrak c$.
The convex hull of a super real closed subring of $A$ is super real closed \cite[Corollary~9.2,~(i)]{Tressl}.

\medskip
\noindent
We can now give:

\begin{proof}[Proof of Lemma~\ref{lem:tressl}]
We view ${}^*\R$ as a super real closed ring in the natural way.
Let $\xi$ be a positive infinite element of ${}^*\R$ and let $\mathcal O$ be the convex hull of the
smallest super real closed subring of ${}^*\R$ containing $\xi$.
Then  $\R_{\operatorname{fin}} \subsetneq \mathcal O$ (since $\xi\in\mathcal O$) and
$\mathcal O \subsetneq {}^*\R$ (by $\mathfrak c^+$-saturation of ${}^*\R$).
The maximal ideal $\mathcal o$ of $\mathcal O$ is $\Upsilon$-radical; extend the residue field $\widehat{\mathcal O}=\mathcal O/\mathcal o$ to a super real closed ring such that the residue map
$\mathcal O\to\widehat{\mathcal O}$ is a morphism of $\mathcal L_{\operatorname{src}}$-structures.
Then $\R$, viewed as  super real closed ring, is existentially closed in $\widehat{\mathcal O}$. 
Note that the ordering of $\R$ is quantifier-free definable in the  super real closed ring $\R$, as
$x\geq 0 \Longleftrightarrow \abs{x}=x$ for each $x\in\R$.
Every restricted $C^\infty$-function $f\colon \R^n\to\R$ is quantifier-free definable in the $\mathcal L_{\operatorname{src}}$-structure $\R$ (since $f\upharpoonright I^n$ has an extension to a continuous function
$\R^n\to\R$, by Tietze Extension); moreover, in the  super real closed ring $\widehat{\mathcal O}$, the same quantifier-free formula defines the function
$f^{\widehat{\mathcal O}}$ introduced above.
All of this now easily implies that $\R_{\mathcal C^\infty}\preceq_1\widehat{\mathcal O}$.
\end{proof}

\noindent
It might be interesting to isolate a tameness property of $\R_{\mathcal C}$,  weaker than o-minimality, which
guarantees $\R_{\mathcal C}\preceq\widehat{\mathcal O}$ (in the language $\mathcal L_{\mathcal C}$) for {\it every}\/ convex subring $\mathcal O$ of ${}^*\R$.

\bibliographystyle{amsplain}

\end{document}